\documentclass[hidelinks,onefignum,onetabnum]{siamart250211}

\date{}

\usepackage{amsmath}
\usepackage{amssymb}
\usepackage{amsfonts}
\usepackage{graphicx}
\usepackage{enumerate}
\usepackage{bm}
\usepackage{hyperref}
\hypersetup{
	colorlinks=true,
	linkcolor=blue,
	filecolor=magenta,      
	urlcolor=blue,
        citecolor=blue
}

\usepackage{xcolor}

\usepackage[sort&compress]{natbib}

\usepackage{listings}
\usepackage{url}
\usepackage[utf8]{inputenc}

\usepackage{color}
\usepackage{courier}
\usepackage{upquote}
\usepackage{framed}
\usepackage{listings}

\usepackage{lipsum}
\usepackage{epstopdf}
\usepackage{algorithmic}
\ifpdf
  \DeclareGraphicsExtensions{.eps,.pdf,.png,.jpg}
\else
  \DeclareGraphicsExtensions{.eps}
\fi

\newsiamremark{remark}{Remark}
\newsiamremark{hypothesis}{Hypothesis}
\crefname{hypothesis}{Hypothesis}{Hypotheses}
\newsiamthm{claim}{Claim}
\newsiamremark{fact}{Fact}
\crefname{fact}{Fact}{Facts}

\headers{SBP operators for general function spaces: Optimal nodes}{N. Hale, C. Harley, P. Nchupang, and J. Nordstr\"om}

\title{Summation-by-parts operators for general\\ function spaces: Optimal nodes}

\author{Nicholas Hale\thanks{Department of Mathematical Sciences, Stellenbosch University, Stellenbosch, South Africa, 7600.}
\and Charis Harley\thanks{Department of Electrical and Electronic Engineering Science, University of Johannesburg, Auckland Park 2006, Johannesburg, South Africa.}
\and Prince Nchupang\footnotemark[1]
\and Jan Nordstr\"om$^\#$\thanks{Department of Mathematics, Link\"{o}ping University, SE-581 83 Link\"{o}ping, Sweden\endgraf\hspace*{0.2pt}$^\#$Department of Mathematics and Applied Mathematics, University of Johannesburg, Auckland Park 2006, Johannesburg, South Africa}}

\usepackage{amsopn}

\ifpdf
\hypersetup{
  pdftitle={Summation-by-parts operators for general function spaces: Optimal nodes},
  pdfauthor={N. Hale, C. Harley, P. Nchupang, and J. Nordstr\"om}
}
\fi

\newcommand{\be}{\begin{equation}}
\newcommand{\ee}{\end{equation}}

\usepackage{placeins}
\begin{document}
\maketitle
%%%%%%%%%%%%%%%%%%%%%%%% ABSTRACT %%%%%%%%%%%%%%%%%%%%%%%%%
\begin{abstract}
Gauss--Lobatto quadrature nodes and weights are optimal for closed summation-by-parts (SBP) formulations based on polynomial approximation spaces in the sense that for a prescribed function space they yield an SBP operator of minimal dimension. We show that the same principle extends to general (possibly non-polynomial) function spaces: an associated generalised Gauss--Lobatto quadrature provides the optimal nodes and weights for the SBP formulation. We present an algorithm for computing these quadrature rules, demonstrate their accuracy and efficiency across a range of function spaces, and illustrate their use in solving initial boundary value problems.
\end{abstract}

%%%%%%%%%%%%%%%%%%%%%%%% ARTICLE %%%%%%%%%%%%%%%%%%%%%%%%%%

% REQUIRED

% REQUIRED

\section{Introduction}\label{sec:intro}
Summation-by-parts (SBP) formulations~\cite{DELREYFERNANDEZ2014171,Svärd201417} combined with weak imposition of boundary conditions ensure semi-discrete stability for well-posed linear initial boundary value problems (IBVPs)~\cite{Kreiss1970277,Gustafsson95,Kreiss04}. The SBP framework emulates the principle of continuous integration-by-parts (IBP) at the discrete level, which is the first crucial step in deriving energy estimates and energy stability for semi-discretized IBVPs~\cite{Nordström2017365}. The second step is to impose boundary conditions weakly, either via the simultaneous approximation term (SAT) method~\cite{Carpenter1994220,Carpenter1999341} or as numerical fluxes~\cite{gassner2013skew,kopriva2021,Kopriva2017314}.

Originally developed in 1974 by Kreiss and Sherer~\cite{KreissSherer}, the SBP technique was rediscovered in the mid 90's~\cite{strand1994,Olsson19951035,Olsson19951473}. This development provided a systematic way to achieve semi-discrete stability for high-order finite difference approximations. The SBP formulation has subsequently been found in (or extended to) almost all existing spatial approximation methods, e.g., in the finite volume \cite{nordstrom2012weak,nordstrom2003finite}, spectral element \cite{carpenter2014entropy,carpenter1996spectral}, flux reconstruction \cite{castonguay2013energy,huynh2007flux}, discontinuous Galerkin \cite{gassner2013skew,hesthaven1996stable,kopriva2021}, and continuous Galerkin schemes \cite{abgrall2020analysis,abgrall2021analysis}.
The SBP-SAT framework has also been extended to the temporal domain~\cite{nordstrom2013,Lundquist201486,Nordström2016A1561}, which allows for fully discrete stable and accurate schemes, even for non-linear IBVPs~\cite{Lundquist16,nordstrom2019}. 

Traditionally, SBP operators
are constructed to be exact for polynomials up to a certain degree. 
However, for some IBVPs, polynomials might not be the best choice. Instead, other approximation spaces could be more suitable, particularly when some a priori knowledge of the solution is available. Glaubitz \textit{et al.}~\cite{glaubitz2023summation} show both how to construct SBP operators for general function spaces $\mathcal{F}$, denoted FSBP operators, and demonstrate that these can significantly improve convergence rates in various applications. The FSBP technique has since been extended to approximate problems with second derivatives~\cite{Glaubitz20242nd}, multi-dimensional problems~\cite{Glaubitz2023multi-d}, and used to improve stability for radial basis function approximations~\cite{Glaubitz2024}. 

%\begin{figure}[ht!]
%     \centering
%         \includegraphics[width=.49\textwidth]{sol_2nd.png}
%         \includegraphics[width=.49\textwidth]{norm_2nd.png}
 %   \caption{(Left) Analytical and numerical solutions of an oscillatory linear-advection problem using 2nd finite difference and trigonometric-based SBP operators.~For the trigonometric bases case,~we used both optimal and equispaced nodes.~(Right) Convergence as the number of subdomains is increased. The continuous and discrete analysis are presented in Section \ref{subsection:contprob_1}}.
%     \label{fig1_triq}
%\end{figure}

In~\cite[Corollary 4.6]{glaubitz2023summation}, Glaubitz \textit{et al}.\ provide a necessary and sufficient condition for the existence of a {\em diagonal-norm} FSBP operator,  namely the existence of a positive quadrature formula, exact for $\mathcal{G} = (\mathcal{F}\mathcal{F})' = \{(fg)':f,g\in\mathcal{F}\}$ (see Theorem~\ref{thm:G2024} below).
For a given function space $\mathcal{F}$, such a quadrature formula can (under some mild restrictions on the elements of $\mathcal{F}$) be obtained by taking sufficiently many grid points and solving a constrained least-squares problem~\cite{glaubitz2024optimization}. 
However, to make the SBP method efficient, one seeks to minimize the number of required grid points and still satisfy this property. This is the focus of the present manuscript, where we demonstrate that when $\mathcal{G} = (\mathcal{FF})'$ forms a {\em Tchebyshev system} then a {\em generalised Gauss--Lobatto quadrature} (GGLQ) rule on $\mathcal{G}$ gives the optimal nodes and weights; see Section~\ref{sec:prelims} for definitions.
%\footnote{Note that Gautcshi~\cite{gautschi04} introduces ``GGLQs'' in a different context. There, the rule is still Gaussian with respect to polynomial spaces, but may involve evaluating higher-order derivatives of the integrand at the endpoints of the interval. Here, the quadratures are generalised in the sense that they apply to non-polynomial function spaces, as per~\cite[\S 2.3.3]{gautschi81}.}

As an example,~in Figure~\ref{fig1_triq} we present the analytical and numerical solutions of an oscillatory linear-advection problem solved using 4th-order polynomial and trigonometric-based SBP operators (see Section \ref{subsection:contprob_1} for details). The right panel of Figure~\ref{fig1_triq} shows that with equally-spaced discretisation points the trigonometric-based SBP scheme is slightly more accurate than the polynomial-based one. However, we also see that a more significant improvement can be made by an optimal choice of nodes, while still maintaining the stability properties guaranteed by the SBP-SAT framework. Finding these otimal nodes is the subject of this manuscript. %The details of the continuous and discrete analysis for the IBVP governing the solution in Figure \ref{fig1_triq} are presented in Section \ref{subsection:contprob_1}.

%\begin{figure}[t!]
%     \centering
 %        \includegraphics[width=.49\textwidth]{sol_4th.png}
 %        \includegraphics[width=.49\textwidth]{norm_4th.png}
 %   \caption{Left: Analytical and numerical solutions of an oscillatory linear-advection problem using 4th-order polynomial and trigonometric-based SBP operators (see Section~\ref{subsection:contprob_1} for details). For the polynomnial basis equally-spaced points are used, while both equispaced and optimal nodes are used for the trigonometric basis.~Right: Convergence as the number of sub-domains (and hence total discretisation size, $N$) in the discontinuous Galerkin discretisation is increased. We see that while solution space tailored for the solution gives improved accuracy, a suitable space combined with optimal nodes is better still.}.
  %   \label{fig1_triq}
%\end{figure}
%\FloatBarrier
\begin{figure}
     \centering
\includegraphics[width=.48\textwidth]{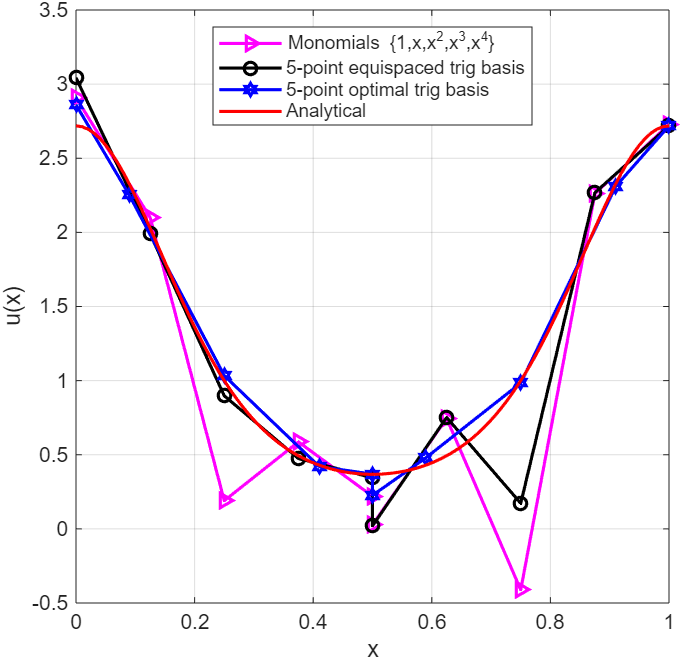}
\includegraphics[width=.49\textwidth]{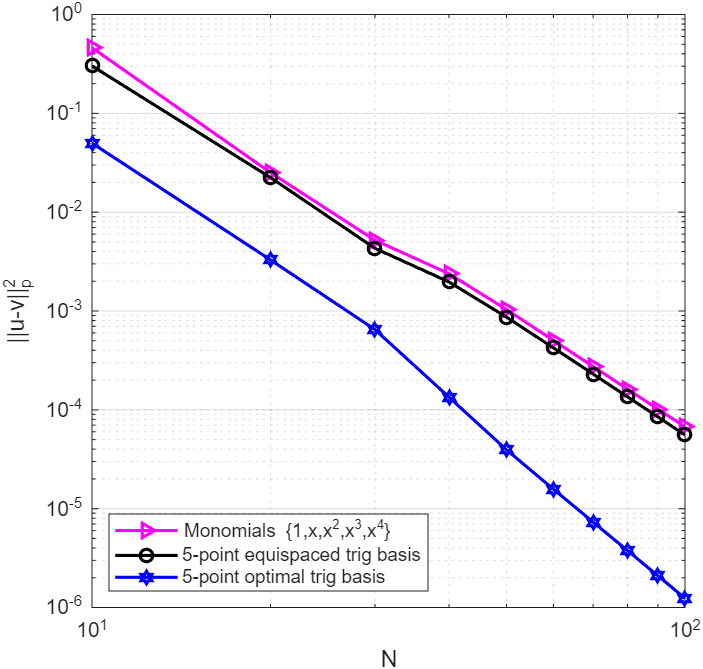}
    \caption{Left: Analytical and numerical solutions of an oscillatory linear-advection problem using 4th-order polynomial and trigonometric-based SBP operators %(see Section~\ref{subsection:contprob_1} for details). 
    For the polynomnial basis equally-spaced points are used, while both equispaced and optimal nodes are used for the trigonometric basis.~Right: Convergence as the number of sub-domains (and hence total discretisation size, $N$) in the discontinuous Galerkin discretisation is increased. We see that while a solution space tailored for the solution gives improved accuracy, a suitable space combined with optimal nodes is better still.}
     \label{fig1_triq}
\end{figure}
\FloatBarrier

This manuscript makes three contributions.
First, building on the existence characterisation of diagonal-norm FSBP operators in~\cite{glaubitz2023summation}, we connect the problem of finding \emph{optimal} nodes (in the sense of minimal operator dimension) to the computation of generalised Gauss and Gauss--Lobatto quadrature rules for the function space $\mathcal{G}=(\mathcal{F}\mathcal{F})'$.
Second, when $\mathcal{G}$ forms a Tchebyshev system, we show that the associated generalised Gauss quadrature (GGQ) and generalised Gauss--Lobatto quadrature (GGLQ) rules provide optimal nodes and positive weights for (open and closed) diagonal-norm FSBP operators.
Third, we present a novel algorithm for computing GGLQ rules based on those of Ma \emph{et al.}~\cite{ma1996generalized} and Yarvin--Rokhlin~\cite{yarvin1998} for computing GGQ rules, and we demonstrate the resulting efficiency gains in multi-domain SBP discretisations of representative IBVPs.

The remainder of the manuscript is organised as follows. 
Section~\ref{sec:prelims} collects the definitions and background needed to connect diagonal-norm FSBP operators to positive quadrature rules.
Section~\ref{sec:GGLQ} develops the GGQ/GGLQ framework for general function spaces, including existence/uniqueness results under Tchebyshev-system assumptions and algorithms for computing the corresponding nodes and weights.
Section~\ref{sec:applics} demonstrates the impact of these optimised operators in multi-domain discretisations of two model IBVPs.
We conclude in Section~\ref{sec:conc}.

\section{Preliminaries}\label{sec:prelims}
\subsection{FSBP operators}\label{sec:def-SBP}
Following Glaubitz {\em et al.}~\cite{glaubitz2023summation}, we define an $\mathcal{F}$-based summation-by-parts (FSBP) operator as follows.
\begin{definition}[FSBP operators]\label{def:2_1}
Let $\mathcal{F} \subset C^1([a,b])$ be a finite dimensional function space {and $\bm{x} = [x_1,\ldots,x_n]^\top \subset [a,b]$ be a set of $n$ distinct nodes}.\footnote{For the remainder of the paper, we assume that all nodes are ordered, distinct, and contained within $[a,b]$, i.e.,  $a\le x_1 < x_2 < \ldots < x_n \le b$.} An operator $D = P^{-1}Q$ is an {\em $\mathcal{F}$-based summation-by-parts (FSBP) operator} if
\begin{enumerate}[i.]\setlength{\itemsep}{-0.25em}
    \item $Df(\textbf{x}) = f'(\textbf{x})$ for all $f \in \mathcal{F}$,
    \item $P$ is a symmetric positive definite matrix,~and
    \item $Q+Q^T = B = \text{diag}(-1,0,\hdots,1).$
\end{enumerate}
An FSBP operator has a {\em diagonal norm} if the matrix $P$ is diagonal and it is {\em closed} if $a, b\in\bm{x}$. 
\end{definition}
\begin{definition}[Optimal FSBP operator]
For a given function space, $\mathcal{F} \subset C^1([a,b])$, an $n\times n$ FSBP operator is {\em optimal} if there is no FSBP operator of dimension less than $n$. {\em Optimal nodes} for the function space are nodes $\bm{x}$ corresponding to an optimal FSBP operator.
\begin{remark}
Note that, for a given function space, the optimal nodes, and hence the optimal FSBP operator, need not be unique---see Section~\ref{subsection:examples}. 
\end{remark}
\end{definition}
\begin{remark}
    Note that if $\mathcal{F} = \mathbb{P}_d$, the space of polynomials of degree $\le d$, then Definition 2.1 recovers the classical SBP framework. In this sense, FSBP operators extend SBP methods from polynomial approximation to general finite-dimensional function spaces. Note also that the optimal nodes for closed polynomial-based SBP operators are known to be Gauss--Lobatto quadrature nodes; see the discussion in Section~\ref{subsec:gqns} below.
\end{remark}
Consider two functions $\phi,\psi\in C^{1}[a,b]$ 
%such that $\phi\psi_x$ and $\phi_x\psi$ are integrable on $[a,b]$  
and denote by $\bm{\phi} = \phi(\bm{x})$ and $\bm{\psi} = \psi(\bm{x})$ vectors of these functions evaluated at the nodes $\bm{x}$. From Definition~\ref{def:2_1} we have that 
    \begin{equation}\label{eqn:discreteIBP}
 \bm{\phi}P(D\bm{\psi}) = \bm{\phi}^\top Q \bm{\psi} = \bm{\phi}^\top \left(B-Q^T\right)\bm{\psi} = 
 \bm{\phi}\bm{\psi}\big|_{1}^{n} - (D\bm{\phi})^\top P \bm{\psi}.
\end{equation}
Thus, we see that an FSBP operator discretely mimics the IBP identity, namely, 
\begin{equation}\label{eqn:ctsIBP}
    \int_a^b \phi \psi_x\,dx = (\phi \psi)\big|^b_a  - \int_a^b \phi_x \psi\,dx.
\end{equation}
As mentioned in Section~\ref{sec:intro}, this {\em discrete IBP} or {\em summation-by-parts} (SBP) property is crucial for deriving energy stable semi-discretized IBVPs~\cite{Nordström2017365}.

Comparing~(\ref{eqn:discreteIBP}) and~(\ref{eqn:ctsIBP}), we can see that the matrix $P$ is the discrete analogue of the definite integral operator on $[a,b]$. For a diagonal norm FSBP, the diagonal entries can therefore be interpreted as weights in a positive $n$-point quadrature rule with $P = \text{diag}(\bm{w})$.
\begin{definition}[Quadrature rule]\label{def:quadrule}
An {\em $n$-point quadrature rule} takes the form
\[
    I_{\bm{x},\bm{w}}[g] = \sum_{k=1}^nw_kg(x_k),
\]
where $\bm{x} = [x_1,\ldots, x_n]^\top$ are {\em nodes} and $\bm{w} = [w_1,\ldots,w_n]^\top$ are {\em weights}. A quadrature rule $I_{\bm{x},\bm{w}}$ is said to be {\em ${\mathcal{G}}$-exact} with respect to the {\em weight function} $\omega(x)$ if it integrates each element of $\mathcal{G}$ exactly, i.e., 
\[
\int_{a}^b g(x)\omega(x)\, dx = I_{\bm{x},\bm{w}}, \quad \forall g\in{\mathcal{G}}.
\]
A quadrature rule $I_{\bm{x},\bm{w}}$ is {\em positive} if $w_k > 0$ for all $k = 1, \ldots, n$ and {\em closed} if $x_1 = a$ and $x_n = b$. 
\end{definition}%

The following result
%, due to Glaubitz {\em et al.}, 
gives a necessary and sufficient condition for the existence of an FSBP operator.
\begin{theorem}{\cite[Corollary 4.6]{glaubitz2023summation}}\label{thm:G2024}
Let $\mathcal{F}\subset C^{1}[a,b]$ be a finite-dimensional function space that forms linearly independent vectors when evaluated at the node points. Then there exists an $\mathcal{F}$-based SBP operator $D = P^{-1}Q$ with a
positive definite diagonal-norm matrix $P=\text{diag}(\bm{w})$ if and only if there exists a positive $(\mathcal{FF)'}$-exact 
 quadrature rule $I_{\bm{x},\bm{w}}$. %Here $(\mathcal{FF})' ={(fg)':f,g\in\mathcal{F}\}.$
\end{theorem}

\noindent This result reduces the problem of constructing FSBP operators to the problem of finding positive quadrature rules exact for $(\mathcal{F}\mathcal{F})'$. Thus, our focus will be (i) identifying function spaces $\mathcal{F}$ for which such quadrature rules exist, and (ii) developing practical methods to compute them.

\subsection{Gauss quadrature and node selection}\label{subsec:gqns}
The following definitions generalise those of standard Gauss quadrature; see~\cite{cheng1999,gautschi81, huybrechs2022} and Section~\ref{sec:GGLQ}.

\begin{definition}[Gauss and Gauss--Lobatto]\label{def:gaussian}
An $m$-point quadrature rule $I_{\bm{x},\bm{w}}$ is called Gaussian if it is exact for $2n$ functions ${g_1,\dots,g_{2n}}$ with respect to a positive {\em weight function} $\omega(x)$, where $m=n$, i.e., 
\[
\int_{a}^{b} g_j(x) \omega(x)\, dx = I_{\bm{x},\bm{w}}[g_j], \qquad j = 1, \ldots, 2n.
\]
If a rule $I_{\bm{x},\bm{w}}$ is exact for the same $2n$ functions ${g_1,\dots,g_{2n}}$, again with respect to $\omega(x)$,
 but the quadrature rule is closed (i.e., includes $a$ and $b$) and $m = n+1$, then it is of {\em Gauss--Lobatto}-type.
\end{definition}%

In classical Gauss quadrature, one takes $\mathcal{G}={\mathbb{P}}_{2n-1}$, i.e., the space of polynomials of degree $2n-1$ or less. When $\omega(x) = 1$ and $[a,b] = [-1,1]$, this gives rise to Gauss--Legendre quadrature, where the nodes are given by roots of the degree $n$ Legendre polynomial and the weights are positive~\cite{gautschi81}. %\footnote{Other finite intervals can be handled via an affine transformation.} %It is well known that an $n$-point Gauss--Legendre quadrature (GLQ) rule has positive weights and is exact for polynomials of degree $2n-1$ or fewer~\cite{gautschi81}. 
Therefore, in the case of an SBP operator for a polynomial basis of degree $n$, $\mathcal{F} = {\mathbb{P}}_n$, for which $\mathcal{G} = (\mathcal{F}\mathcal{F})' = {\mathbb{P}}_{2n-1}$, an $n$-point Gauss--Legendre quadrature rule satisfies the conditions of Theorem~\ref{thm:G2024}. % \footnote{Any other finite interval can be handled by an affine change of variables.} 
However, in SBP-based methods, closed quadrature formulas are typically preferred, leading to the use of $n+1$ point Gauss--Legendre--Lobatto quadrature formulas~\cite{gassner2013skew,kopriva2021,Kopriva2017314}. An $n+1$ point Gauss--Legendre--Lobatto quadrature is positive and exact for polynomials of degree $2n-1$~\cite[Section 2.1]{gautschi81}, therefore satisfying the conditions of Theorem~\ref{thm:G2024}. 
By contrast, Newton--Cotes rules on equally spaced nodes can suffer from negative weights (e.g., for $n=9$ or $n\ge 11$~\cite{berstein1937formules}), violating condition {\em ii.}\ of Definition~\ref{def:2_1}. Even if the weights are positive, an equally spaced rule will typically require $2n$ points to exactly integrate a polynomial of degree $2n-1$ (as fixing the $2n$ nodes leaves only the $2n$ degrees of freedom in the weights). For large $n$, this discrepancy---$2n$ points for uniform nodes versus $n+1$ points for optimal nodes---can be considerable, making Gauss--Legendre--Lobatto-based SBP formulations significantly more efficient and, in fact, optimal.

In the next Section we discuss how these ideas can be extended to non-polynomial function spaces and how the associated {\em generalised Gauss quadrature rules} (GGQs) can be constructed.

\section{Generalised Gauss quadrature}\label{sec:GGLQ}%
In the polynomial case, Gauss--Legendre and Gauss--Legendre--Lobatto quadrature rules underpin the construction of SBP operators. For more general function spaces, a natural analogue arises: generalised Gauss quadrature (GGQ) rules. Since an $n$-point quadrature has $2n$ free parameters (nodes and weights), it is reasonable to expect exactness for $2n$ functions. This observation, which goes back to Markov~\cite{markov1898} and later developed by Karlin \& Studden~\cite{karlin196} and Krein~\cite{kreuin1959ideas}, motivates the modern theory of GGQs~\cite{ma1996generalized}. Whereas the theory of classical Gauss quadratures is closely connected to that of orthogonal polynomials, the theory of generalised Gauss quadratures is tied to Tchebyshev systems.

\begin{definition}[Tchebyshev system]
The sequence of functions $g_1, \ldots, g_n \in C^{1}[a,b]$ is referred to as a {\em Tchebyshev system} on a finite interval $[a, b]$ if and only if each of them is real-valued and continuous on $[a, b]$ and the determinants
\[
    \left|\begin{matrix}
    g_1(x_1) & g_2(x_1) & \ldots & g_{n}(x_1)\\
    g_1(x_2) & g_2(x_2) & \ldots & g_{n}(x_2)\\
    \vdots&\vdots&\vdots&\vdots\\
    g_1(x_n) & g_2(x_n) & \ldots & g_{n}(x_n)
    \end{matrix}\right|
\]
are non-zero for any set of distinct points $\{x_k\}_{k=1}^n\subset[a,b]$.
\end{definition}

\begin{theorem}%Karlin \& Studden~\cite{karlin196}%
\cite[Theorem~2.1]{ma1996generalized}\label{thm:KS} Suppose that the functions $g_1, \ldots, g_{2n}$ constitute a Tchebyshev system on the interval $[a, b]$. Then there exists a unique $n$-point quadrature
rule that is Gaussian with respect to the functions $g_1, \ldots, g_{2n}$. Furthermore, all the weights $w_1, \ldots, w_n$, of the quadrature are positive.
\end{theorem}

\begin{corollary}\label{corr:FSBPexist}
    If $\mathcal{G} = (\mathcal{FF})'$ forms a Tchebyshev system on $[a,b]$, then an $\mathcal{F}$-based SBP operator exists. Furthermore, the optimal $\mathcal{F}$-based SBP operator has nodes and weights given by the GGQ rule for the function space $\mathcal{G}$.
\end{corollary}
\begin{proof}[Proof of Corollary \ref{corr:FSBPexist}]
    The proof follows from Theorems~\ref{thm:G2024} and~\ref{thm:KS}. In particular, if $\mathcal{G} = (\mathcal{FF})'$ forms a Tchebyshev system on $[a,b]$ then from Theorem~\ref{thm:KS} there exists a positive quadrature rule that is ${\mathcal{G}}$-exact. From Theorem~\ref{thm:G2024} it follows that a diagonal-norm FSBP operator exists. The optimality follows from the fact that the rule in Theorem~\ref{thm:KS} is Gaussian on $\mathcal{G}$.
\end{proof}%

\noindent Theorem~\ref{thm:KS} above is, more precisely,  a corollary of various results from Karlin \& Studden~\cite{karlin196} given by~\cite{ma1996generalized} (although the latter credits the former with the result, and similar results can be found in~\cite{kreuin1959ideas} and~\cite{markov1898}). What Karlin \& Studden prove is (i) the existence of {\em upper and lower principal representations of moment vectors from the moment space} induced by the Tchebyshev set $g_1, \dots, g_{2n}$~\cite[Theorem 3.1]{karlin196} and (ii) that the {\em lower} principal representation corresponds to an $n$-point quadrature rule which is exact for $g_1, \dots, g_{2n}$~\cite[Theorem 8.1]{karlin196}. Importantly, their results also show that the {\em upper} principal representation results in a {\em closed} $n+1$ point rule that is also exact for $g_1, \dots, g_{2n}$, i.e., a generalized Gauss--Lobatto quadrature (GGLQ) rule. (See~\cite[Section~2]{huybrechs2022} for a summary.) %Karlin \& Studden give an example in~\cite[Corollary 8.1a]{karlin196} where the function space consists of spline functions. T
This can be summarised in the following result.

%\red{TODO: Remove this paragraph? Furthermore,  Karlin \& Studden discuss what amounts to a GGLQ formula in~\cite[Corollary 8.1b]{karlin196}, for a function space consisting of spline functions. They refer to this as a `Radau' formula but, in more modern terminology, this phrase typically refers to a quadrature rule that employs precisely one of the end points. They do, in fact, discuss the existence of Radau rules for Tchebyshev systems of odd dimension, i.e., $g_1, \dots, g_{2n-1}$, where again the rules correspond to upper and lower {principal} representations. Generalised Gauss Radau quadratures are utilised in~\cite{huybrechs2022}.} 

\begin{theorem}[Karlin \& Studden~\cite{karlin196}]\label{thm:new} Suppose that the functions $g_1, \ldots, g_{2n}$ constitute a Tchebyshev system on the interval $[a, b]$. Then there exists a unique $(n+1)$-point Gauss--Lobatto quadrature with respect to the $2n$ functions $g_1, \ldots, g_{2n}$. Furthermore, all the weights $w_1, \ldots, w_{n+1}$, of the quadrature are positive.
\end{theorem}
\begin{proof}
    Follows from results in \S3 and \S8 of \cite[\S3]{karlin196} and the discussion above. 
    \end{proof}
    
\begin{remark}
Karlin \& Studden discuss what amounts to a GGLQ formula in~\cite[Corollary 8.1b]{karlin196}, for a function space consisting of spline functions. They refer to this as a `Radau' formula but, in more modern terminology, this phrase typically refers to a quadrature rule that employs precisely one of the end points.%\todoinline{Remove this remark if we are short of space.}
\end{remark}

\begin{corollary}\label{corr:FSBPexist2}
    If $\mathcal{G} = (\mathcal{FF})'=\{g_1, \ldots, g_{2n}\}$ forms a Tchebyshev system on $[a,b]$, then a {closed} $\mathcal{F}$-based SBP operator exists. Furthermore, the optimal closed $\mathcal{F}$-based SBP operator has nodes and weights given by the $(n+1)$-point GGLQ rule for the function space $\mathcal{G}$.
\end{corollary}
\begin{proof}
    Follows from Theorems~\ref{thm:G2024} and~\ref{thm:new}, similarly to Corollary~\ref{corr:FSBPexist}.
    \end{proof}

\noindent Thus, both GGQ and GGLQ rules provide the optimal building blocks for FSBP operators whenever $(\mathcal{F}\mathcal{F})'$ forms a Tchebyshev system. The practical question, however, is how to compute such rules in concrete function spaces. While the computation of GGQ rules have received some attention, e.g., \cite{bremer2010, huybrechs2022, ma1996generalized}, the literature on GGLQ rules is comparatively lacking---a gap we address in Section~\ref{subsubsec:modified}.

\subsection{Computation of GGQs}\label{subsubsec:Ma_etal}
Recall from Definition~\ref{def:gaussian} that, for a given function space $\mathcal{G}$ of dimension $2n$ on $[a,b]$, spanned by $\{g_1, g_2, \ldots, g_{2n}\}$, the defining property of a GGQ is that 
\begin{equation}\label{eqn:nonlinear}
\int_a^b g_j(x)\omega(x)\, dx = \sum_{k=1}^{n}w_k g_j(x_k), \quad j = 1, \ldots, 2n.    
\end{equation}
\begin{remark}
    In principle, for the construction of FSBP operators, we need consider only the weight function $\omega(x) = 1$. However, as discussed below, it will be beneficial to use modified weight functions during the continuation process used in the construction procedure described below.
\end{remark}
%If the integrals on the left-hand side can be computed (or at least approximated numerically), t
Equation~(\ref{eqn:nonlinear}) results in a $2n\times2n$ nonlinear system of equations to solve for the unknown nodes $x_k$ and weights $w_k$, $k = 1, \ldots, n$. One can attempt to solve this nonlinear system directly via, say, the Newton--Raphson method, but in general (i) obtaining initial guesses is difficult and (ii) the computation is expensive since the Jacobians are dense.   Ma {\em et al}.~\cite{ma1996generalized} were the first to resolve the issues above by (i) using a continuation approach to obtain suitable initial conditions and (ii) reformulating~(\ref{eqn:nonlinear}) to an equivalent system which, when $\mathcal{G}$ forms an {\em extended Hermite system} can be solved via a quadratically convergent $n$-dimensional quasi-Newton method with a diagonal Jacobian.

\begin{definition}[Hermite--Vandermonde matrix]\label{def:hvmatrix}
For a finite sequence of functions $g_1(x)$, \ldots, $g_{2n}(x)$ and set of $n$ distinct points, $\{x_k\}_{k=1}^n$, the $2n\times2n$ Hermite--Vandermonde matrix is given by
\[
    V = \left[\begin{matrix}
    g_1(x_1) & g_2(x_1) & \ldots & g_{2n}(x_1)\\
    g_1(x_2) & g_2(x_2) & \ldots & g_{2n}(x_2)\\
    \vdots&\vdots&\vdots&\vdots\\
    g_1(x_n) & g_2(x_n) & \ldots & g_{2n}(x_n)\\
    g'_1(x_1) & g'_2(x_1) & \ldots & g'_{2n}(x_1)\\
    \vdots&\vdots&\vdots&\vdots\\
    g'_1(x_n) & g'_2(x_n) & \ldots & g'_{2n}(x_n)
    \end{matrix}\right].
\]
\end{definition}
\begin{definition}[Hermite system]
A finite sequence of functions $g_1$, \ldots, $g_{2n}$ is said to form a {\em Hermite system} on $[a,b]$ if $g_k\in C^1[a,b]$ and $\det(V)\not=0$ for any set of distinct points $\{x_k\}\subset[a,b]$, where $V$ is the corresponding Hermite--Vandermonde matrix.

\end{definition}
\begin{definition}[Extended Hermite system]
A finite sequence of functions $g_1$, \ldots, $g_{2n}$ is said to form an {\em extended Hermite system} on $[a,b]$ if it is both a Tchebyshev system and a Hermite system.
\end{definition}

Since~\cite{ma1996generalized}, there have been noteworthy 
contributions that improve upon the original algorithm~\cite{cheng1999, yarvin1998} or provide alternative approaches~\cite{bremer2010, huybrechs2022, huybrechs2008}. In particular, the algorithm of Bremer {\em et al.}~\cite{bremer2010} is advantageous in that it places less restrictive conditions on the function space to which it can be applied and, in situations where GGQs do not exist, can still produce ``nearly Gaussian quadratures''. Huybrechs~\cite{huybrechs2022, huybrechssoftware} presents a more robust algorithm for Tchebyshev systems and even briefly discusses GGLQs. %{Huybrechs' algorithm is available in a Julia package~\cite{huybrechssoftware}, but unfortunately, the implementation does not possess GGLQ functionality.} % (although it does support generalised Gauss--Radau quadratures (GGRQ)).
However, since the algorithm of Ma {\em et al}.\ is the most readily adaptable to the computation of GGLQs, we focus our attention there.

The algorithm of Ma {\em et al}.\ proceeds as follows. Consider a finite sequence of functions, ${\mathcal{G}} = g_1, \ldots , g_{2n}$, which form an extended Hermite system, and a positive weight function $\omega(x)$. For a given set of points, $\bm{x} = \{x_j\}_{j=1}^n$, in $[a,b]$ construct a {\em Hermite--Lagrange basis}, $\{\sigma_1, \ldots, \sigma_n, \eta_1, \ldots, \eta_n\}$, of $\mathcal{G}$ so that
\[
\sigma_i(x)  = \sum_{j=1}^n\alpha_{ij}g_j(x), \qquad \eta_i(x)  = \sum_{j=1}^n\alpha_{ij}g_j(x), 
\]
and
\begin{equation}\label{eqn:mahermlag}
\sigma_i(x_k) = 0, \quad \sigma'_i(x_k) = \delta_{ik}, \quad \eta_i(x_k) = \delta_{ik}, \quad \eta'_i(x_k) = 0, \qquad i,k = 1, \ldots, n.
\end{equation}
Since we assume that ${\mathcal{G}}$ forms an extended Hermite system, this is always possible~\cite[Lemma 4.2]{ma1996generalized}.\footnote{The basis can be found by simply inverting the Hermite--Vandermonde matrix of Definition~\ref{def:hvmatrix}, evaluated at the points $\bm{x}$, since det$(V)\not=0$.} One can then show---see~\cite[Theorem~4.1]{ma1996generalized}---that the points $\bm{x}$ are Gaussian with respect to $\omega(x)$ and ${\mathcal{G}}$ if and only if
\begin{equation}\label{eqn:sigmacondition}
\int_a^b \sigma_i(x)\omega(x)\, dx = 0, \quad i=1,\ldots, n,
\end{equation}
in which case the weights are given by 
\[
w_i = \int_a^b \eta_i(x)\omega(x)\,dx, \quad i=1,\ldots, n.
\]
This leads to a quasi-Newton iteration of the form 
\[
x_k^{[\ell+1]} = x_k^{[\ell]} + \frac{\int_a^b \sigma^{[\ell]}_k(t)\omega(x)\,dx}{\int_a^b \eta^{[\ell]}_k(x)\omega(x)\, dx}, \quad k=1,\ldots, n, \ \ell = 0, 1, \ldots,
\]
where $\{\sigma^{[\ell]}_k\}_{k=1}^n$ and $\{\eta^{[\ell]}_k\}_{k=1}^n$ are the basis functions~(\ref{eqn:mahermlag}) induced by the points $\bm{x}^{[\ell]}$.
For a sufficiently good initial guess, $\bm{x}^{[0]}$, this iteration will converge quadratically to a vector $\bm{x}$  such that~(\ref{eqn:sigmacondition}) is satisfied~\cite[Corollary 4.2]{ma1996generalized}.

To address the problem of initial guesses, Ma {\em et al}.\ use a continuation approach, whereby the GGQ nodes are successively calculated for a function space $G(t) = tG + (1-t)G^*$,  $t\in\{t_0 = 0, t_1, t_2, \ldots, t_N = 1\}$, where $G^*$ is some function space for which the GGQ nodes are known, for example polynomials, $\mathbb{P}_{2n-1}$. A drawback of this approach is that it requires that $G(t)$ forms an extended Hermite system for all $t\in(0, 1)$, not just $t = 0$ and $t = 1$. Since this is generally not guaranteed, a more robust alternative is proposed by Yarvin and Rokhlin~\cite{yarvin1998}, namely continuation in the integration measure. In particular, if $\omega(x)$ is the desired weight function, they propose obtaining GGQs for the weight functions 
\begin{equation}\label{eqn:continuation}
    \omega(x,t) = t\omega(x) + (1-t)\sum_{k=1}^{n}\delta(x-c_j),
\end{equation}
for $t\in\{t_0 = 0, t_1, t_2, \ldots, t_N = 1\}$. The $t_k$ can  be fixed a priori or selected in an adaptive manner. When  $t = 0$ the weights are simply given by $x_j = c_j$. To choose $c_j$, they use the interlacing properties of GGQ nodes---see~\cite[Theorem 3.2]{karlin196}---and compute GGQs of increasing degree. % and taking $c^{[j+1]}_1 = x^{[j]}_1, c^{[j+1]}_{2:n-1} = (x^{[j]}_{1:n-2}-x^{[j]}_{2:n-1})/2, c^{[j+1]}_n = x^{[j]}_{n-1}$. 
 We employ this approach in our computations, taking $\omega(x) = 1$. 
 
The approach of Huybrechs~\cite{huybrechs2022} also employs continuation. There, GGQs of increasing degree are computed by continuously transforming between upper and lower principal representations, resulting in GGQ and generalised Gauss--Radau quadrature (GGRQ) rules, respectively. When applied to Tchebyshev sets, this requires only the solution of a sequence of piecewise-smooth, monotonic minimisation problems. Since the deformations are continuous, this results in (essentially) arbitrarily good initial guesses, making convergence of the Newton iteration more robust. However, modifying the approach to compute GGLQs is more involved and requires a different continuation mechanism. We leave this for future investigations.

\begin{remark}\label{remark:GSBP}
As discussed in Section~\ref{subsec:gqns} above, most SBP operators are based on closed quadrature rules, but there are exceptions~\cite{Fernández2014214,Lundquist2024,Nordström2017451}. The construction of optimal {\em open} GSBP operators could directly use the existing GGQ theory presented in this section without the extension to GGLQs in the next.
%while later efforts worth mentioning involves~\cite{carpenter1996spectral}   
%    Similar matrices arise in the derivation of the {\em other} kind of GGLQs, mentioned in footnote 1.
\end{remark}

\subsection{Computation of GGLQs}\label{subsubsec:modified}
Modifying the algorithm of Section~\ref{subsubsec:Ma_etal} to enable computation of GGLQ nodes and weights is relatively straightforward. As before, consider a given function space $\mathcal{G} = \text{span}\{g_1, g_2, \ldots, g_{2n}\}$ which forms a Tchebyshev system on $[a,b]$, and a positive weight function $\omega(x)$. Now, consider a closed $n+1$-point grid of distinct points, $\bm{x} = \{x_k\}_{k=0}^{n}\subset[a,b]$ with $x_0 = a$ and $x_n = b$, and construct a {\em quasi-Hermite--Lagrange basis}, $\{\sigma_1, \ldots, \sigma_{n-1}, \eta_0, \ldots, \eta_{n}\}$, so that 
\[
\sigma_i(x_k) = 0, \quad \eta_i(x_k) = \delta_{ik}, \qquad  k = 0, \ldots, n
\]
and
\begin{equation}\label{eqn:modmadiff}
\sigma'_i(x_k) = \delta_{ik} \quad \eta'_i(x_k) = 0, \qquad k = 1, \ldots, n-1,
\end{equation}
where in both cases $i$ runs over $1,\ldots, n-1$ for the $\sigma_i$ and $0, \ldots, n$ for the $\eta_i$.
The key differences from the algorithm in Section~\ref{subsubsec:Ma_etal}  are that: our grid now consists of $n+1$ points, rather than $n$;  the derivative conditions in~(\ref{eqn:modmadiff}) are not enforced at the first and last nodes as they are in~(\ref{eqn:mahermlag}); there are $n-1$ terms in $\{\sigma_1, \ldots, \sigma_{n-1}\}$ and $n+1$ terms in $\{\eta_0, \ldots, \eta_{n}\}$ compared to $n$ in both previously. However, this still results in a square system of $2n$ linear equations and, if the matrix
\[
    \widehat V = \left[\begin{matrix}
    g_1(x_0) & g_2(x_0) & \ldots & g_{2n}(x_0)\\
    g_1(x_1) & g_2(x_1) & \ldots & g_{2n}(x_1)\\
    %g_1(x_2) & g_2(x_2) & \ldots & g_{2n}(x_2)\\
    \vdots&\vdots&\vdots&\vdots\\
    g_1(x_n) & g_2(x_n) & \ldots & g_{2n}(x_n)\\
    g'_1(x_1) & g'_2(x_1) & \ldots & g'_{2n}(x_1)\\
    \vdots&\vdots&\vdots&\vdots\\
    g'_1(x_{n-1}) & g'_2(x_{n-1}) & \ldots & g'_{2n}(x_{n-1})
    \end{matrix}\right].
\]
can be inverted (taking the place of~(\ref{def:hvmatrix})), then we may expand the $\sigma_i$ and $\eta_i$ as linear combinations of the $g_j$ (and vice-versa), as in the original algorithm. 

One can readily adapt the theorems and proofs of \cite[Section 4]{ma1996generalized} to show that if $x_0 = a$ and $x_n = b$, then $I_{\bm{x}, \bm{w}}$ is a generalised Gauss--Lobatto rule on ${\mathcal{G}}$ with respect to $\omega(x)$ if and only if the points $x_1,\ldots,x_{n-1}$ are such that 
\[
\int_a^b \sigma_i(x)\omega(x)\, dx = 0, \quad i=1,\ldots, n-1,
\]
in which case the weights are
\[
w_i = \int_a^b \eta_i(x)\omega(x)\,dx, \quad i=0,\ldots, n.
\]
Furthermore, as before, the interior points can be found by the locally quadratically convergent quasi-Newton iteration
\[
x^{[\ell+1]}_k = x_k^{[\ell]} + \frac{\int_a^b \sigma_k^{[\ell]}(x)\omega(x)\,dx}{\int_a^b \eta_k^{[\ell]}(x)\omega(x)\, dx}, \quad k=1,\ldots, n-1, \ \ell = 0, 1, \ldots.
\]
\noindent The continuation scheme~(\ref{eqn:continuation})  of~\cite{yarvin1998} can be employed without modification. Examples are given in Section~\ref{subsection:examples}.

\subsection{Constructing FSBP operators}
Once optimized nodes $\bm{x}$ and the associated positive weights $\bm{w}$ have been obtained,~the next task is to construct the differentiation matrix $D = P^{-1} Q$ that satisfies Definition \ref{def:2_1}. Here, $P$ is simply the diagonal matrix with diagonal entries $\bm{w}$. To construct $Q$ such that it is almost skew-symmetric, we follow~\cite{Glaubitz2023multi-d,Glaubitz20242nd,glaubitz2023summation,Glaubitz2024}. From $Q+Q^T = B$, we write %$Q$ as
$
Q = \frac{1}{2} B + S, \text{ where } S = \frac{1}{2} (Q-Q^T)\label{Qdef}
$ is the skew-symmetric component. We consider the consistency condition $P^{-1}Q F = F_x$ where $F$ is the Vandermonde-like matrix with columns consisting of the basis functions evaluated at points $\bm{x}$ and $F_x$ is its corresponding derivative. Next, using $\eqref{Qdef}$, we rewrite the consistency condition  as
\[
SF = PF_x - \frac{1}{2} BF.
\]
The only unknown here is $S$ and it can be obtained by solving the constrained least squares problem $F^TS = R^T$, where $R =  PF_x - \frac{1}{2}BF$ and $S$ is constrained to be skew-symmetric. For more details on the construction procedure and resulting operators, see Glaubitz \textit{et al.}~\cite{Glaubitz2023multi-d,Glaubitz20242nd,glaubitz2023summation,Glaubitz2024}.

\subsection{Numerical considerations}\label{subsec:numerical}

We employ the Chebfun package in our implementation of the algorithm above, which provides a powerful and flexible framework for numerical computations involving general function spaces~\cite{chebfun}. Its design---representing functions as piecewise polynomial interpolants---makes it especially well suited to problems such as ours, which requires high-accuracy numerical orthogonalisation,  differentiation, and integration.

One key advantage of using Chebfun is its ability to {adaptively} determine piecewise discretisations of given functions. As noted in~\cite{ma1996generalized}, this is crucial when working with general function spaces, especially those with localized features, discontinuities in derivatives, or rapidly changing behaviour. Rather than fixing a discretisation \emph{a priori}, Chebfun refines the representation until a desired accuracy is achieved, thereby enabling robust numerical integration and differentiation.

Another major benefit is that, given a basis for $\mathcal{F}$, Chebfun allows us to {automatically} compute a basis for the space $(\mathcal{FF})'$. Specifically, Chebfun's singular value decomposition {\tt svd} implementation enables the extraction of an orthogonal basis of $(\mathcal{FF})'$. % This greatly simplifies the workflow and reduces the likelihood of algebraic or transcription errors in manually computing these products and derivatives.
Working with an orthogonal basis helps mitigate ill-conditioning when constructing and inverting the Hermite--Vandermonde matrices, and is recommended in~\cite{ma1996generalized}. %Orthogonalisation is especially helpful when the functions in $\mathcal{G}$ are nearly linearly dependent or poorly scaled, as it improves the stability of both interpolation and least-squares procedures. 
A similar procedure is employed in~\cite{yarvin1998}.

Finally, Chebfun can {adaptively} compute accurate numerical approximations of the integrals and derivatives over $\mathcal{G}$, as  needed in the quasi-Newton iterations for node updates. % or for determining quadrature weights. % Since Chebfun constructs function representations to machine precision by default, its quadrature rules (e.g., Clenshaw--Curtis) provide highly accurate approximations to integrals, even for non-smooth or oscillatory functions. This is particularly important when computing expressions like $\int_a^b \sigma_i(x)\, dx$ and $\int_a^b \eta_i(x)\, dx$, which directly affect convergence and accuracy of the GGQ/GGLQ node construction.
Taken together, these features make Chebfun an ideal tool for implementing and extending the GGQ and GGLQ algorithms in the context of general function spaces and FSBP operator construction. A MATLAB code to generate these nodes, weights, and the corresponding SBP differentiation matrix is available at~\cite{code}.

\subsection{Examples}
\label{subsection:examples}

\mbox{}
\newline
\textbf{Example I:} In~\cite{glaubitz2023summation}, Glaubitz {\em et al.} present an example for the exponential approximation space 
\begin{equation}\label{eqn:example1}
\mathcal{F} = 
%\mathcal{E}_2 = 
\text{span}\{1, x, e^x\},
\end{equation}
on $[0, 1]$, resulting in 
\[
    \mathcal{G} = (\mathcal{F}\mathcal{F})'
    %(\mathcal{E}_2\mathcal{E}_2)' 
    = \text{span}\{1, x, e^x, xe^x, e^{2x}\}.
\]
Using the least-squares approach they obtain the 5-point equally-spaced quadrature formula
$$
\begin{array}{c|ccccccccc}
	\bm{x} & 0 & 0.25 & 0.5 & 0.75 & 1\\
    \bm{w} & 0.0759763872 & 0.3620888878 & 0.1244746618 & 0.3608784643 & 0.0765815990
\end{array},
$$
giving rise to the differentiation matrix 
$$
\setlength{\arraycolsep}{3.25pt} % default is 6pt
\!\!\!\! D = \left[\begin{array}{rrrrrrrr}
\\[-.5em]
	-6.528983553&   8.521455370& -0.479376527 & -2.489678844 & 0.976583554\\
	-1.808328128&            0&  0.890963460 &  1.451385592 & -0.534020924\\
	0.294930875& -2.583092176&             0 & 2.569553027 & -0.28139173\\
	0.526565695& -1.446533768& -0.883330737 &           0  &  1.803298810\\
	-0.984362811&  2.536533493 & 0.461013497 & -8.594176227 &   6.580992049\\[.5em]
    \end{array}\right].$$
If we instead use the procedure described in Section~\ref{subsubsec:modified}, we are able to compute a 4-point Lobatto-type grid with positive weights that is $\mathcal{G}$-exact, namely%\footnote{MATLAB code to generate these nodes, weights, and the corresponding SBP differentiation matrix are available at~\tocite{}.}
$$
\begin{array}{c|cccccccc}%
	\bm{x} & 0 & 0.2956452974 & 0.7423537958  & 1\\
    \bm{w} &   0.0914828668 & 0.4341375639&	0.3987262252 & 0.0756533441
\end{array},
$$
giving rise to the differentiation matrix 
$$
D = 
\left[\begin{array}{rrrrrr}\\[-.5em]
-5.465504277&  7.365125959& -2.802901094&   0.903279412\\
-1.552003083&            0&  2.142484824& -0.590481741\\
0.643091453& -2.332761387&            0&   1.689669934\\
-1.092279411&  3.388486098& -8.905299859&   6.609093173\\[.5em]
\end{array}\right].
$$

%\begin{remark}
In this example, $\cal{G}$ has only 5 elements. The derivations in the previous sections were based on the fact that $\mathcal{G}$ had an even number of basis functions.\footnote{Otherwise, in the language of Tchebyshev sets, this would result in a principal representation of {\em half-integer index}, giving rise to a Radau-type rule~\cite[Corollary 3.1]{karlin196}, typically undesirable for SBP constructions.} Therefore, when $\dim\mathcal{G}$ is odd, we augment $\mathcal{G}$ with an additional linearly independent function; in this case $g_6(x) = x^2$. Since there are many other possible such functions we could have used to augment $\mathcal{G}$, each resulting in a different set of nodes and weights, the optimal nodes for the original function space are not unique. 

Admittedly, the difference here between a five-point equally-spaced rule and a four-point optimal rule is not that impressive. However, for larger function spaces the gains will be more significant.

\begin{remark}
In~\cite{glaubitz2024optimization} a $4\times4$ operator for the function space~(\ref{eqn:example1}) based on a uniform grid is presented, namely\footnote{In~\cite{glaubitz2024optimization} these values are given to only 4 decimal places. Here we have calculated the $\bm{w}$ and $D$ to higher precision using our own implementation of the optimisation procedure described in that same manuscript.}
$$
\begin{array}{c|ccccccccc}
	\bm{x} & 0 & 0.3333333333 & 0.6666666666 & 1\\
    \bm{w} & 0.1413079925 &     0.3300510668  &    0.4159300208 &    0.1126686223	\end{array},
$$
and
$$
D = \left[\begin{array}{rrrrr}
\\[-.5em]
  -3.538370274  & 3.606205594 & 0.4025982272  &-0.4704099266\\
   -1.543960084  &                 0 &  1.631806089 & -0.08783997038\\
  -0.1367786512 & -1.294879700 &                  0 &  1.431657018\\
   0.5899839814 &  0.2573181010 & -5.285137255 &  4.437792793\\[.5em]
    \end{array}\right].$$
One can verify that properties {\em (ii)} and {\em (iii)} from Definition~\ref{def:2_1} are satisfied by these nodes and weights, however {\em (i)} is far from satisfied, giving a maximum error of around $1.4\times10^{-4}$ when differentiating the basis functions. Whether this operator can be viewed as an FSBP operator for the function space~(\ref{eqn:example1}) is therefore questionable.
\end{remark}

\textbf{Example II:} Consider the function space
\begin{equation}\label{eqn:example2}
    \mathcal{F} = \{J_\nu(x), \nu = 0, 1, \ldots, 9\},
\end{equation}
on the interval $[0, 25]$. Here $J_\nu(x)$ are Bessel functions, which arise often in the context of radially-symmetric waves and flows, and so an FSBP operator built on such a basis may be beneficial in these settings~\cite{korenev2002bessel}.  Determining $\mathcal{G}$ explicitly here is complicated, so we use the automatic Chebfun SVD-based procedure described above, with which one finds that $\mathcal{G}$ is spanned by a space of dimension 48. Using the modified Ma {\em et al}.\ routine, we find the positive 25-point GGLQ rule shown in Table~\ref{table:example2}.

\begin{table}[ht]
\centering
\begin{tabular}{c c | c c  c c}
$\bm{x}$ & $\bm{w}$ &
$\bm{x}$ & $\bm{w}$ \\\hline
0            & 0.04674109541 & 12.14501738  & 1.586926028 \\
0.1708613339 & 0.2857413813  & 13.78368019  & 1.634993556 \\
0.5661606569 & 0.5014293903  & 15.30246428  & 1.373778538 \\
1.166443811  & 0.6963842628  & 16.63807436  & 1.392918441 \\
1.954074975  & 0.8750817726  & 18.11924712  & 1.503655935 \\
2.905578711  & 1.022716109   & 19.49873798  & 1.226087151 \\
3.993567909  & 1.150946736   & 20.68293671  & 1.214279320 \\
5.198004650  & 1.252002006   & 21.87902576  & 1.094192400 \\
6.495279513  & 1.347451304   & 22.83987635  & 0.8957497727 \\
7.897105635  & 1.446405823   & 23.70798388  & 0.7768641260 \\
9.334639871  & 1.393772578   & 24.35339492  & 0.5561938824 \\
10.680527    & 1.350167740   & 24.80987539  & 0.3170472565 \\[-.6em]
    {\small \vdots}     &     {\small \vdots}          & 25.0         & 0.05847348825 \\
\end{tabular}
\caption{The 25-point quadrature nodes $\bm{x}$ and associated weights $\bm{w}$ corresponding to the function space~(\ref{eqn:example2}). For the same function space, an equally-spaced grid requires over 100 quadrature points to construct a valid FSBP operator. }\label{table:example2}
\end{table}

Conversely, using an equally-spaced grid, the optimisation procedure of~\cite{glaubitz2024optimization} required 100 points to generate an FSBP operator for the same function space.

\section{Applications to IBVPs}\label{sec:applics}
To illustrate the efficiency of our development, we will consider two examples of IBVPs,~both in a multi-domain setting.

\subsection{Advection equation}
\label{subsection:contprob_1}
Consider the constant coefficient problem 
\begin{align}
u_t + au_x &= 0,& \quad 0 \leq x \leq 1, \quad t > 0, \nonumber \\
u  &= g_L(t),& \quad  x = 0, \quad t \geq 0, \label{eq:prob1_eq1} \\
u &= f(x),&\quad 0 \leq x \leq 1, \quad t = 0, \nonumber
\end{align}
where $a> 0$.~The boundary data $g_L$ and initial data $f$ are compatible such that a smooth solution exists.
\par 
We apply the energy method \cite{Gustafsson95, Kreiss1970277} by multiplying the equation in $\eqref{eq:prob1_eq1}$ with the solution $u$ and then integrating over the domain, to get the energy rate
\begin{eqnarray}\label{eq:prob1_rate}
\frac{d}{dt}||u||^2 = au^2 \big|^0_1  = ag^2_L - au^2(1,t),
\end{eqnarray}
which leads to a bound in terms of data by time-integration.~Here,~$||u ||^2 = \int^1_0 u^2 dx$ is the $L_2$ norm.
\par In this example and the one to follow we utilise a multi-element construction, where we subdivide the domain into multiple sub-domains. On each of these sub-domains we utilise a fixed number of nodes so that for an increasing number of elements we have an increasing number of nodes. For the sake of clarity, we illustrate how this is done by partitioning the domain $x\in [0,1]$ into two sub-domains, i.e. $x_L =[0,x_i]$ and $x_R\in [x_i,1]$ where $x_i$ is an intermediate point.~These sub-domains are discretized with $N+1$ and $M+1$ points,~respectively.~Further,~we define ${\bf u}$ and ${\bf v}$ as the numerical approximations of $u$ and $v$ on the left and right sub-domains.~As presented in \cite{Carpenter1999341}, the semi-discrete approximation of (\ref{eq:prob1_eq1}) on the left and right sub-domains is
\begin{eqnarray}\label{eq:prob1_eq3a}
    {\bf u}_t + a D_L {\bf u} &=& \sigma^LP_L^{-1}{\bf e}^L_N(u_N-v_0) + \tau^L P_L^{-1}{\bf e}^L_0(u_0-g_L),\\ \label{eq:prob1_eq3b}
{\bf v}_t + aD_R {\bf v} &=& \sigma^R P_R^{-1}{\bf e}^R_0\left(v_0-u_N\right)
\end{eqnarray}
where the SBP difference operators, $D_L, D_R$, are defined in Definition \ref{sec:def-SBP}. The interface conditions and left boundary condition are implemented weakly by the SAT terms on the right hand side. We apply the energy method by multiplying (\ref{eq:prob1_eq3a}) and (\ref{eq:prob1_eq3b}) with ${\bf u}^T P_L$ and ${\bf v}^T P_R$ respectively, adding the transposes, using the relations in Definition \ref{def:2_1} and summing up. That leads to the following discrete energy rate
\begin{align}
\frac{d}{dt}\left( ||{\bf u}||^2_{P_L} + ||{\bf v}||^2_{P_R}\right)    =  ag^2_L - av_N^2-a(u_0-g_L)^2-a(v_0-u_N)^2, \label{energy_rate2}
\end{align}
%\leq  ag^2_L +(-a+2\sigma^L)(u_N-v_0)^2,
where we have used the penalty coefficient $\tau_L = -a $.
%such that the left outer boundary terms are bounded by data.
Further,~we set $\sigma^R = \sigma^L - a$ so that we have a conservative scheme and a decreasing energy rate if $\sigma^L \leq \frac{a}{2}$. We choose $\sigma^L = 0$ in equation (\ref{energy_rate2}) which leads to an energy rate similar to (\ref{eq:prob1_rate}) but with damping terms.~Lastly,~the temporal integration leads to an energy estimate and hence stability.~Here,~$||\textbf{u}||^2_P = \textbf{u}^T P\textbf{u}$ is the discrete $L_2$-equivalent norm. 

We will compute and compare approximate solutions to equation (\ref{eq:prob1_eq1}) for the following constructions of the derivative operator:
\begin{itemize}
\item $4^{th}$-order polynomial-based SBP operator.
\item Trigonometric-based SBP operator with equispaced nodes using the Glaubitz et al.'s \cite{glaubitz2023summation} approach.
\item Trigonometric-based SBP operator with optimised nodes using the approach described in Section \ref{subsubsec:modified}.
\end{itemize}
For the trigonometric function space,~we consider the bases $\{1,\sin(k\pi x),\cos(k\pi x)\}$ which gives $2k + 1$ bases when we choose $k\in [1,2]$. We assess the accuracy of the approach via the method of manufactured solution (MMS) \cite{MMS1} with the 
%following analytical 
solution
\begin{align}
v(x) = e^{\sin(2\pi x)},
\end{align}
which satisfies (\ref{eq:prob1_eq1}) with periodic data.~The exact and approximated solutions are shown in Figure \ref{fig1_triq} (left).~Further,~we present the solution error norm ($||u - v||^2_P$) in Figure \ref{fig1_triq} (right) for both approaches for an increasing number of nodes, where we subdivide the domain into multiple sub-domains. 
%On each of these sub-domains we have 5 equispaced nodes for the one approximation and 4 optimal nodes for the other
As shown in Figure \ref{fig1_triq}, and discussed above, the solution obtained with operators based on trigonometric bases and optimal nodes is significantly more accurate than the solution produced by the other operators.
%via the operators constructed with optimised trigonometric bases converges faster than the equispaced counterpart.}

\subsection{Advection-diffusion equation}
\label{subsection:contprob_2}
Consider the constant coefficient problem
\begin{align}\nonumber
u_t + au_x -\varepsilon u_{xx} &= 0,& \quad 0 \leq x \leq 1,& \quad t> 0, \\ \nonumber
au -\varepsilon u_x &= g_L& \quad x = 0, & \quad t \geq 0, \\ \label{prob2_eq1} 
\varepsilon u_x &= g_R& \quad x = 1, &\quad t \geq 0, \\ \nonumber
u &= f(x),&\quad  0 \leq x \leq 1,& \quad t = 0, 
\end{align}
where $a, \varepsilon > 0$ are the wave speed and diffusion constant,~respectively. At the left boundary we impose the Robin boundary condition,~$au-\varepsilon u_x= g_L$,~and at the right boundary a Neumann boundary condition,~$\varepsilon u_x = g_R$, where $g_L$ and $g_R$ denote boundary data and $f$ initial data. 
\par To solve (\ref{prob2_eq1}) and avoid using second derivative operators (for which we presently have no theory) we employ the Local Discontinuous Garlerkin (LDG) method \cite{JCAM2011}. Using $\varphi = u_x$ we rewrite the governing equation in $\eqref{prob2_eq1}$ in first order form as
\begin{align}
u_t +au_x - \varepsilon \varphi_x &= 0,  \label{eq8}   \\ 
-\varepsilon u_x + \varepsilon \varphi &= 0. \nonumber 
\end{align}
We apply the energy method to $\eqref{eq8}$ with boundary conditions from (\ref{prob2_eq1}) to obtain
\begin{align}\label{eq:rate2}
%\frac{d}{dt}||u||^2+2\varepsilon ||\varphi||^2 = \frac{1}{a}\left[g_L^2 -(\varepsilon\varphi)^2 \right]\big|_{x=0}-\frac{1}{a}\left[(au-\varepsilon \varphi)^2 -g_R^2 \right]\big|_{x=1} \leq \frac{g_L^2}{a} + \frac{g_R^2}{a},
\frac{d}{dt}||u||^2+2\varepsilon ||\varphi||^2 = \frac{1}{a}\left[g_L^2 + g_R^2  \right]-\frac{1}{a}\left[(\varepsilon\varphi)^2\big|_{x=0} + (au-\varepsilon \varphi)^2\big|_{x=1} \right] ,
\end{align}
which bounds the energy rate.
% \cite{JCAM2011}
\par Similar to the previous example in Section \ref{subsection:contprob_1}, we partition the domain $x\in [0,1]$ into two sub-domains.~The semi-discrete approximation of the system \eqref{eq8} on the left domain leads to
\begin{align}\nonumber
\textbf{u}_t + a D_L \textbf{u} - \varepsilon D_L \bm{\varphi} &=  \sigma_1^LP_L^{-1}{\bf e}^L_N(u_N-v_0) +  \sigma_2^LP_L^{-1}{\bf e}^L_N \left(\varphi_N - \eta_0\right) \\ &+ \tau_L P_L^{-1}{\bf e}^L_0 (au_0-\varepsilon \varphi_0 - g_L),  \label{eq:prob2_eq3a}  \\
  - \varepsilon D_L {\boldsymbol u} + \varepsilon {\boldsymbol \varphi} &= \sigma_3^LP_L^{-1}{\bf e}^L_N(u_N-v_0) +  \sigma_4^LP_L^{-1}{\bf e}^L_N \left(\varphi_N - \eta_0\right). \nonumber
\end{align}
Meanwhile,~on the right domain, we have
\begin{align}\nonumber
\textbf{v}_t + aD_R \textbf{v} -\varepsilon D_R \bm{\eta} &= \sigma_1^R P_R^{-1}{\bf e}^R_0\left(v_0-u_N\right) + \sigma_2^R P_R^{-1}{\bf e}^R_0\left(\eta_0 - \varphi_N\right) \\
 &+ \tau_R P_R^{-1}{\bf e}^R_M(\varepsilon \eta_M-g_R), \label{eq:prob2_eq3b} \\
  -\varepsilon D_R \textbf{v} + \varepsilon \bm{\eta} &= \sigma_3^R P_R^{-1}{\bf e}^R_0\left(v_0-u_N\right) + \sigma_4^R P_R^{-1}{\bf e}^R_0\left(\eta_0 - \varphi_N\right). \nonumber 
\end{align}
The energy method applied to (\ref{eq:prob2_eq3a}) and (\ref{eq:prob2_eq3b}) gives
\begin{align}\nonumber
\frac{d}{dt}\left(||{\boldsymbol u}||^2+||{\boldsymbol v}||^2 \right) + 2\varepsilon ||{\boldsymbol \varphi}||^2+2\varepsilon ||{\boldsymbol \eta} ||^2 =& \frac{1}{a} \left[g_L^2 +g_R^2 \right]-\frac{1}{a}\left(au_0 -g_L \right)^2  \label{energy_rateLDG2} \\ 
-& \frac{1}{a}\left( av_M - g_R \right)^2 - a(u_N - v_0)^2 \\
-&\varepsilon (\varphi_N-\eta_0)^2,\nonumber
\end{align}
with inserted values of penalty coefficients $\tau_L = \tau_R =  -1$.
%, such that the left and right boundary terms are bounded by data. 
We also used the following stability conditions from \cite{JCAM2011}
\begin{align}\label{eq:conditions}
\sigma_1^R = -a+\sigma_1^L, \quad \sigma_2^R =\varepsilon +  \sigma_2^L ,\quad \sigma_2^L = -\varepsilon -\sigma_3^L, \quad \sigma_3^R = \varepsilon + \sigma_3^L , \quad \sigma_4^R=\sigma_4^L.
\end{align}
and choose $\sigma_1^L = 0$ and $\sigma_4^L = -\frac{\varepsilon}{2}$. Therefore, $\eqref{energy_rateLDG2}$ with $\eqref{eq:conditions}$ leads to an energy rate
which is similar to $\eqref{eq:rate2}$ but with additional damping terms. %We use $\sigma_1^L \leq \frac{a}{2}$ and $\sigma_4^L \leq 0$ to ensure that it is decreasing.

\par Next,~we compare approximate solutions to equation (\ref{prob2_eq1}) using (\ref{eq:prob2_eq3a})--(\ref{eq:prob2_eq3b}) with derivative operators based on exponential basis functions with equispaced and optimised nodes, respectively.~We also implement polynomial bases with equispaced nodes and Gauss Lobatto nodes as baseline solutions for comparison. The MMS is utilised to assess the accuracy of the approach 
%taken 
with the 
%following manufactured 
solution
\begin{align}
v(x,t) = \left( \frac{e^{ax/\varepsilon}-1}{e^{a/\varepsilon}-1} \right)e^{0.1t}.
\end{align}
%and obtain the forcing function,~boundary and initial data by inserting this solution into $\eqref{eq8}$.~
To capture the exponential behaviour of the solution,~we use the basis $\{1,x,e^{ax/\varepsilon}\}$. 
\par ~In Figure \ref{fig1.1} the accuracy of the solutions obtained 
for the case $\varepsilon = 0.1$ is shown. We observe that the introduction of an FSBP operator including an exponential function improves 
the convergence rates. In addition, we see that the use of optimal nodes 
improves the convergence orders even more for both polynomial based SBP operators and exponentially based FSBP operators. The best result by a wide margin is  obtained for the FSBP operator with an exponential basis and optimal nodes.

%\begin{figure}[ht!]
%     \centering
%      \includegraphics[width=.5\textwidth]{u 005.png}\hspace{-1.25cm}
 %     \includegraphics[width=.5\textwidth]{norms epsilion 005.png}
 %  \caption{A boundary layer-like problem \textcolor{red}{with $\varepsilon = 0.05$} approximated using exponential and polynomial basis FSBP operators}. Left: Exact solution compared with the SBP solution. Right: Convergence as the number of nodes is increased. \label{fig1.1}
%\end{figure}

\begin{figure}[ht!]
     \centering
      \includegraphics[width=.5\textwidth]{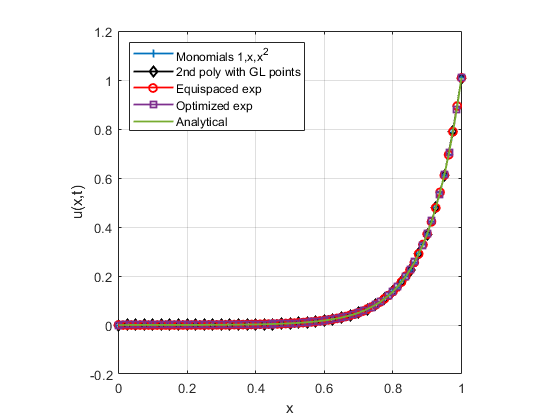}\hspace{-1.25cm}
      \includegraphics[width=.5\textwidth]{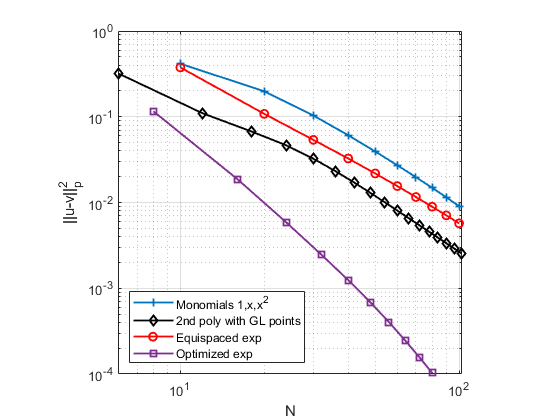}
   \caption{A boundary layer-like problem with $\varepsilon = 0.1$ approximated using exponential and polynomial basis FSBP operators. For both cases we use equidistant and optimal nodes. Left: Exact solution compared with the SBP solution. Right: Convergence as the number of nodes is increased.}\label{fig1.1}
\end{figure}
\FloatBarrier
\noindent 

\section{Conclusion}\label{sec:conc}

We have studied the construction of diagonal-norm FSBP operators for general (non-polynomial) approximation spaces by exploiting the equivalence between such operators and positive quadrature rules exact for $\mathcal{G}=(\mathcal{F}\mathcal{F})'$.
When $\mathcal{G}$ forms a Tchebyshev system, the theory of generalised Gauss quadrature implies existence, uniqueness, and positivity of the corresponding GGQ and GGLQ rules, and therefore yields \emph{optimal} node sets in the sense of minimal operator dimension.

To build a closed diagonal-norm FSBP operator for a prescribed space $\mathcal{F}$, one should compute a positive GGLQ rule for $\mathcal{G}=(\mathcal{F}\mathcal{F})'$ and then construct $D=P^{-1}Q$ from these nodes and weights.
This directly transfers the efficiency of Gauss--Lobatto nodes in polynomial SBP to the setting of general function spaces.

Our numerical examples demonstrate that the resulting optimised nodes reduce the number of required grid points substantially.
For instance, for the Bessel function space in Example~II we obtain a positive 25-point GGLQ-based construction, whereas an equally spaced grid required 100 points in the optimisation-based approach of~\cite{glaubitz2024optimization}.

The main practical limitations are (i) reliability in characterising (and numerically working with) the structure of $\mathcal{G}$ and determining when it forms a Tchebyshev system, and (ii) reliability in computing GGLQ rules. 
Promising directions include: robust detection of Tchebyshev-system structure for $\mathcal{G}=(\mathcal{F}\mathcal{F})'$ in finite precision arithmetic;  ``nearly'' GGQ/GGLQ constructions when $\mathcal{G}$ is not a strict Tchebyshev system, together with accuracy/stability guarantees;  and incorporating the ideas of Huybrechs~\cite{huybrechs2022} for more robust computation of GGLQs;
as well as extensions of the present approach to higher-order operators (e.g., second derivatives), bypassing the necessity of LDG formulations.

In summary, generalised Gauss and Gauss--Lobatto quadrature provide a direct route to node selection for FSBP operators, yielding minimal-size stable discretisations whenever the underlying function-space assumptions hold. We also demonstrated the efficiency and accuracy of our improvement when solving initial boundary value problems.

\section{Acknowledgements}\label{sec:ack}
Generative AI (ChatGPT) was used to assist in editing the authors’ text for spelling, grammar, and structure in this work. AI was not used in the generation of any code, numerical results, or figures.

\bibliographystyle{abbrv}
\bibliography{mybib}%

\end{document}